\documentclass[12pt]{amsart}
\usepackage{mathpazo}
\usepackage{hyperref}
\usepackage{bm}
\usepackage{verbatim}
\usepackage[margin=1in]{geometry}
\usepackage{graphicx}
\usepackage[all]{xy}
\usepackage{array}
\usepackage{amssymb,amsmath,latexsym,amsthm}
\newtheorem{theorem}{Theorem}[section]
\newtheorem{lemma}[theorem]{Lemma}
\newtheorem{proposition}[theorem]{Proposition}
\newtheorem{corollary}[theorem]{Corollary}
\newtheorem{predefinition}[theorem]{Definition}
\newenvironment{definition}{\begin{predefinition}\rm}{\end{predefinition}}
\newtheorem{preremark}[theorem]{Remark}
\newenvironment{remark}{\begin{preremark}\rm}{\end{preremark}}
\newtheorem{prenotation}[theorem]{Notation}

\newtheorem{preexample}[theorem]{Example}
\newenvironment{example}{\begin{preexample}\rm}{\end{preexample}}
\newtheorem{preclaim}[theorem]{Claim}

\newtheorem{prequestion}[theorem]{Question}

\newtheorem{preapplication}[theorem]{Application}

\numberwithin{equation}{section}

\newcommand \ZZ {{\mathbb Z}}
\newcommand \QQ {{\mathbb Q}}
\newcommand{\Q}{{\mathbb Q}}
\newcommand \NN {{\mathbb N}}
\newcommand  \FF {{\mathbb F}}

\newcommand \PP {{\mathbb P}^1}

\global\let\hom\undefined
\DeclareMathOperator{\hom}{Hom}

\DeclareMathOperator{\End}{End}

\newcommand \CA {{\mathcal A}}
\newcommand \CM {{\mathcal M}}

\newcommand \CT {{\mathcal T}}

\newcommand \cC {{\mathcal C}}

\newcommand \frf {{\mathfrak f}}

\newcommand \CO{{\mathfrak O}}
\newcommand \p{{\mathfrak p}}

\newcommand \CC {{\mathbb C}}

\newcommand \cf {{\mathfrak f}}

\newcommand \QM {{\QQ[\mu_m]}}

\newcommand \co {{\mathfrak o}}

\newcommand \CJ {{\mathcal J}}
\newcommand \hp {{\langle p\rangle}}

\newcommand \Kd {{K_d}}
\newcommand \Km {{K_m}}

\usepackage[usenames,dvipsnames]{color} 

\title{Newton polygons of cyclic covers of the projective line branched at three points}
\date{}

\author{Wanlin Li}
\address{Department of Mathematics,
University of Wisconsin,
Madison, WI 53706, USA}
\email{wanlin@math.wisc.edu}

\author{Elena Mantovan}
\address{Department of Mathematics,
California Institute of Technology
Pasadena, CA 91125, USA}
\email{mantovan@caltech.edu}

\author{Rachel Pries}
\address{Department of Mathematics, 
Colorado State University, 
Fort Collins, CO 80523, USA}
\email{pries@math.colostate.edu}

\author{Yunqing Tang}
\address{Department of Mathematics,
Princeton University,
Princeton, NJ 08540, USA}
\email{yunqingt@math.princeton.edu}

\begin{document}

\begin{abstract}
We review the Shimura-Taniyama method for computing the Newton polygon of an abelian variety with complex multiplication.
We apply this method to cyclic covers of the projective line branched at three points.
As an application, we produce multiple new examples of Newton polygons 
that occur for Jacobians of smooth curves in characteristic $p$.
Under certain congruence conditions on $p$, these include:
the supersingular Newton polygon for each genus $g$ with $4 \leq g \leq 11$;
nine non-supersingular Newton polygons with $p$-rank $0$ with $4 \leq g \leq 11$; and,
for all $g \geq 5$, the Newton polygon with $p$-rank $g-5$ having slopes $1/5$ and $4/5$.

MSC10: primary 11G20, 11M38, 14G10, 14H40, 14K22; secondary 11G10, 14H10, 14H30, 14H40

\end{abstract}

\maketitle

\section{Introduction}

In positive characteristic $p$, there are several discrete invariants associated with abelian varieties,  
e.g., the $p$-rank, the Newton polygon, and the Ekedahl--Oort type.  
These invariants give information about the Frobenius morphism and 
the number of points of the abelian variety defined over finite fields.  
It is a natural question to ask which of these invariants can be realized by Jacobians of smooth curves.  

For any prime $p$, genus $g$ and $f$ such that $0 \leq f \leq g$, Faber and van der Geer prove in \cite{FVdG}
that there exists a smooth curve of genus $g$ defined over $\overline{\FF}_p$ which has $p$-rank $f$.
Much less is known about the Newton polygon, more precisely, the Newton polygon of the characteristic polynomial of Frobenius. 
For $g=1,2,3$, it is known that every possible Newton polygon occurs for a smooth curve of genus $g$.
Beyond genus $3$, very few examples of Newton polygons are known to occur. 
In \cite[Expectation 8.5.4]{oort05}, for $g \geq 9$,
Oort observed that 
it is unlikely for all Newton polygons to occur for Jacobians of smooth curves of genus $g$.

This project focuses on Newton polygons of cyclic covers of the projective line ${\mathbb P}^1$.  
One case which is well-understood is when the cover is branched at $3$ points, especially when the cover is of prime degree (see \cite{GR}, \cite{honda}, \cite{Weil}, \cite{yui}).
In this case, the Jacobian is an abelian variety with complex multiplication and its 
Newton polygon can be computed using the Shimura--Taniyama theorem. 
In Section \ref{sec_dim0}, we give a survey of this material, including composite degree.  
Although this material is well-known, it has not been systematically analyzed for this application.
We use this method to tabulate numerous Newton polygons having $p$-rank $0$ 
which occur for Jacobians of smooth curves.

We now describe the main results of this paper in more detail.
By \cite[Theorem 2.1]{VdGVdV}, if $p=2$ and $g \in \NN$, then there exists a supersingular curve of genus $g$ defined over 
$\overline{\FF}_2$.  
For this reason, we restrict to the case that $p$ is odd in the following result.
In the first application, we verify the existence of supersingular curves in the following cases.
In Remark \ref{Runlikely}, we explain why the $g=9$, $g=10$ and $g=11$ cases are especially interesting.

\begin{theorem} \label{Tintro1} (See Theorem \ref{Tapp1})
Let $p$ be odd.
There exists a smooth supersingular curve of genus $g$ defined over $\overline{\FF}_p$ in the following cases:\\
$g=4$ and $p \equiv 2 \bmod 3$, or $p \equiv 2,3,4 \bmod 5$;\\
$g=5$ and $p \equiv 2,6,7,8,10 \bmod 11$;\\
$g=6$ and $p \not\equiv 1,3,9 \bmod 13$ or $p \equiv 3,5,6 \bmod 7$;\\
$g=7$ and $p \equiv 14 \bmod 15$ or $p \equiv 15 \bmod 16$;\\ 
$g=8$ and $p \not\equiv 1 \bmod 17$; \\
$g=9$ and $p \equiv 2,3,8,10,12,13,14,15,18 \bmod 19$;\\
$g=10$ and $p \equiv 5,17,20 \bmod 21$;\\
$g=11$ and $p \equiv 5,7,10,11,14,15,17,19,20,21,22 \bmod 23$.
\end{theorem}

The second application is Theorem \ref{TotherNP}: under certain congruence conditions on $p$, 
we prove that nine new Newton polygons that have $p$-rank $0$ but are not supersingular occur 
for Jacobians of smooth curves.   

For context for the third application,
recall that every abelian variety is isogenous to a factor of a Jacobian.  
This implies that every rational number $\lambda \in [0,1]$ occurs as a 
slope for the Newton polygon of the Jacobian of a smooth curve. 
In almost all cases, however, there is no control over the other slopes in the Newton polygon.

In the third application, when $d=5$ or $d=11$, under congruence conditions on $p$, we show 
that the slopes $1/d$ and $(d-1)/d$ occur 
for a smooth curve in characteristic $p$ of arbitrarily large genus with complete control over the other slopes in the Newton polygon.
Namely, under these congruence conditions on $p$ and for all $g \geq d$, 
we prove that there exists a smooth curve of genus $g$ defined over $\overline{\FF}_p$ whose Newton polygon
contains the slopes $1/d$ and $(d-1)/d$ with multiplicity $d$ and slopes $0$ and $1$ with multiplicity $g-d$.
This was proven earlier, for all $p$, when $d=2$ \cite[Theorem 2.6]{FVdG}; 
$d=3$ \cite[Theorem 4.3]{Pr:large}; and $d=4$ \cite[Corollary 5.6]{AP:gen}.

Let $G_{1, d-1} \oplus G_{d-1,1}$ denote the $p$-divisible group with slopes $1/d, (d-1)/d$.

\begin{theorem} \label{Tintro2} (See Theorem \ref{Tapp2})
For the following values of $d$ and $p$ and for all $g \geq d$, 
there exists a smooth curve of genus $g$ defined over $\overline{\FF}_p$
whose Jacobian has $p$-divisible group isogenous
to $(G_{1, d-1} \oplus G_{d-1,1}) \oplus (G_{0,1} \oplus G_{1,0})^{g-d}$:
\begin{enumerate}
\item $d=5$ for all $p \equiv 3,4,5,9 \bmod 11$;
\item $d=11$ for all $p \equiv 2,3,4,6,8,9,12,13,16,18 \bmod 23$.
\end{enumerate}
\end{theorem}

In future work, we determine new results about Newton polygons of curves arising in positive-dimensional special families of cyclic covers
of the projective line.
This work relies on the Newton polygon stratification of PEL-type Shimura varieties. 
Then we attack the same questions for arbitrarily large genera using a new induction argument for Newton polygons 
of cyclic covers of ${\mathbb P}^1$.
We use the Newton polygons found in this paper as base cases in this induction process. 

\subsection*{Organization of the paper} \mbox{ } \\

Section \ref{sec_prelim} contains basic definitions and facts about group algebras, cyclic covers of ${\mathbb P}^1$, and Newton polygons. 

Section \ref{sec_dim0} focuses on the Jacobians of cyclic covers branched at exactly three points.  We review the 
Shimura--Taniyama method for computing the Newton polygon and provide examples. 

Section \ref{sec_table} contains tables of data.

Section \ref{Sapplication} contains the proofs of the three theorems. 

\subsection*{Acknowledgements} \mbox{ } \\
This project began at the \emph{Women in Numbers 4} workshop at the Banff International Research Station.
Pries was partially supported by NSF grant DMS-15-02227.
We thank the referee for the valuable feedback and comments.

\section{Notation and background} \label{sec_prelim} 

\subsection{The group algebra $\QM$}\label{prelim_gpalg} \mbox{ } \\

For an integer $m\geq 2$, let $\mu_m:=\mu_m(\CC)$ denote the group of $m$-th roots of unity in $\CC$. 
For each positive integer $d$, we fix a primitive $d$-th root of unity $\zeta_d=e^{2\pi i/d}\in\CC$.
Let $K_d=\QQ(\zeta_d)$ be the $d$-th cyclotomic field over $\QQ$ of degree $\phi(d)$.

Let $\QM$ denote the group algebra of $\mu_m$ over $\QQ$.
It has an involution $*$ induced by the inverse map on $\mu_m$, i.e., $\zeta^*:= \zeta^{-1}$ for all $\zeta\in\mu_m$.  
The $\QQ$-algebra $\QM$ decomposes as a product of finitely many cyclotomic fields, namely
\[\QM=\prod_{0<d\mid m}K_d.\]
The involution $*$ on $\QM$ preserves each cyclotomic factor $K_d$, and for each $d\mid m$, the restriction of $*$ to $K_d$ agrees with complex conjugation.

Let $\CT$ denote the set of homomorphisms $\tau:\QM\to\CC$.
In the following, we write \[\QM\otimes_\QQ\CC=\prod_{\tau\in\CT} \CC,\] and for each $(\QM\otimes_\QQ\CC)$-module $W$, we write $W=\oplus_{\tau\in\CT} W_\tau$, where $W_\tau$ denotes the subspace of $W$ on which $a\otimes 1\in\QM\otimes_\QQ\CC$ acts as $\tau(a)$.

For convenience, we fix an identification $\CT=\ZZ/m\ZZ$ by defining, for $n\in \ZZ/m\ZZ$, \[\tau_n(\zeta):=\zeta^n, \text{ for all }\zeta\in\mu_m.\]
Note that, for any $n\in\ZZ/m\ZZ$, and $a\in\QM$, \[{\tau}_{-n}(a)=\tau_{n}(a^*)=\overline{\tau_n(a)},\]
where $z\mapsto \overline{z}$ denotes complex conjugation on $\CC$. In the following, we write $\tau_n^*:=\tau_{-n}$.

\begin{remark}
For each $\tau\in\CT$, the homomorphism $\tau:\QM\to\CC$ factors via a projection $\QM\to K_d$, for a unique positive divisor $d$ of $m$. We refer to $d$ as {\em the order of} $\tau$. Indeed, for each $n\in\ZZ/m\ZZ$, the homomorphism $\tau_n$ factors via the cyclotomic field $K_d$ if  and only if $d$ is the exact order of $n$ in $\ZZ/m\ZZ$.
\end{remark}

For each rational prime $p$, we fix an algebraic closure $\QQ_p^{\rm alg}$of $\QQ_p$, and an identification $\CC\simeq \CC_p$, where $\CC_p$ denotes the $p$-adic completion of  $\QQ_p^{\rm alg}$. We denote by
$\QQ_p^{\rm un}$ the maximal unramified extension of $\QQ_p$ in  $\QQ_p^{\rm alg}$, 
and by $\sigma$ the Frobenius of $\QQ_p^{\rm un}$. 

Assume that $p$ does not divide $m$. Then $\QM$ is unramified at $p$ (i.e., the group $\mu_m$ is \'etale over $\ZZ_p$), and, for each $\tau\in\CT$, the homomorphism $\tau:\QM\to \CC\simeq \CC_p$ factors via the subfield $ \QQ_p^{\rm un}\subset\CC_p$. In particular,
 \[\QM\otimes_\QQ \QQ_p^{\rm un}=\prod_{\tau\in\CT} \QQ_p^{\rm un}.\]
There is a natural action of the Frobenius $\sigma$ on the set $\CT$, defined by $\tau\mapsto \tau^\sigma:=\sigma\circ \tau$. 
Then $\tau_n^\sigma=\tau_{pn}$ for all $n\in\ZZ/m\ZZ$. 

We write $\CO$ for the set of $\sigma$-orbits $\co$ in $\CT$.  For each $\tau\in\CT$, we denote by $\co_\tau$ its $\sigma$-orbit. The set $\CO$ is in one-to-one correspondence with the set of primes $\p$ of $\QM$ above $p$. 
We write $\p_\co$ for the prime above $p$ associated with an orbit $\co$ in $\CT$. For each $\sigma$-orbit $\co\in \CT$, the order of $\tau$ is the same for all $\tau\in \co$ and we denote this order by $d_\co$. Let $K_{d_\co, \p_\co}$ denote the completion of $K_{d_\co}$ along the prime $\p_\co$. Then \[\QM\otimes_\QQ\QQ_p=\prod_{\co\in\CO} K_{d_\co, \p_\co}.\]

\subsection{Cyclic covers of the projective line}\label{prelim_curve} \mbox{ } \\

Fix an integer $m\geq 2$, together with a triple of positive integers $a=(a(1),a(2), a(3))$. 
We refer to such a pair $(m,a)$ as a {\em monodromy datum} if
\begin{enumerate}
\item $a(i)\not\equiv 0\bmod m$, for all $i=1, 2, 3$, 
\item $\gcd(m, a(1), a(2), a(3))=1$, 
\item $a(1)+a(2)+a(3)\equiv 0 \bmod m$.
\end{enumerate}

Fix a monodromy datum $(m,a)$. 
The equation 
\begin{equation} \label{Ecurve}
y^m=x^{a(1)} (x-1)^{a(2)}
\end{equation}
defines a smooth projective curve $C=C_{(m,a)}$ defined over $\QQ$. 
The function $x$ on $C$ yields a map $C \to {\mathbb P}^1$,
and there is a $\mu_m$-action on $C$ over $ {\mathbb P}^1$ given by $\zeta\cdot (x,y)=(x,\zeta\cdot y)$ for all $\zeta\in\mu_m$ (more precisely, this action is defined on the base change of $C$ from $\QQ$ to $K_m$).
The curve $C$, together with this $\mu_m$-action, is a $\mu_m$-Galois cover of the projective line $\PP$;
it is branched at $ 0,1,\infty$ and has local monodromy $a(1)$ at $0$, $a(2)$ at $1$ and $a(3)$ at $\infty$.
By the hypotheses on the monodromy datum, for primes $p \nmid m$ the reduction of $C$ at $p$ is a geometrically irreducible curve of genus $g$, where
\begin{equation} \label{Egenus}
g=g(m,a)=1+\frac{m-\gcd(a(1),m)-\gcd(a(2),m)-\gcd(a(3),m)}{2}.
\end{equation}

\begin{remark}
The isomorphism class of the curve $C=C_{(m,a)}$ depends only on the equivalence class of the monodromy datum $(m,a)$, where two monodromy data $(m,a)$ and $(m',a')$ are equivalent if $m=m'$,  and the images of $a,a'$ in $(\ZZ/m\ZZ)^3$ are in the same orbit under the action of $(\ZZ/m\ZZ)^*\times\Sigma_3$, where $\Sigma_3$ is the symmetric group of degree $3$.
\end{remark}

Let $V:=H^1(C(\CC), \QQ)$ denote the first Betti cohomology group of $C$. 
Then $V$ is a $\QM$-module, and there is a decomposition $V\otimes_{\QQ}\CC=\oplus_{\tau \in \CT}V_\tau$. 
In addition, $V$ has a Hodge structure of type $(1,0)+(0,1)$ ,with the $(1,0)$ piece given by $H^0(C(\CC), \Omega^1_{C})$ via the Betti--de Rham comparison. We denote by $V^+$ (resp.\ $V^-)$ the $(1,0)$ (resp.\ $(0,1)$) piece. Both $V^+$ and $V^-$ are $\QM$-modules, so there are decompositions
\[V^+=\oplus_{\tau\in \CT} V^+_\tau \ {\rm and} \  \quad V^-=\oplus_{\tau\in \CT} V^-_\tau.\]

Let $\cf(\tau):=\dim_{\CC} V^+_\tau$.  For any $q\in \QQ$, let $\langle q\rangle$ denote the fractional part of $q$. 
By \cite[Lemma 2.7, Section 3.2]{moonen} (see also \cite{deligne-mostow}),
\begin{equation}\label{DMeqn}
\cf(\tau_n)=\begin{cases} -1+\sum_{i=1}^3\langle\frac{-na(i)}{m}\rangle \text{ if $n\not\equiv 0 \bmod m$}\\
0\text{ if $n\equiv 0 \bmod m$}.\end{cases}
\end{equation}
We call  $\cf=(\cf(\tau_1), \ldots, \cf(\tau_{m-1}))$ the \emph{signature type} of the monodromy datum $(m,a)$.

\begin{remark}\label{rmk_dimV}
Let $n(\tau):=\dim_\CC V_\tau$.
For all $\tau\in \CT$, one sees that $\dim_\CC V^+_{\tau^*}=\dim_\CC V^-_{\tau}$ and thus $\cf(\tau)+\cf(\tau^*)=n(\tau)$.  
Note that $n(\tau)$ depends only on the order of $\tau$, and thus only on the orbit $\co_\tau$.
If $\co \in \CO$, we sometimes write $n(\co)=n(\tau)$, for any $\tau\in\co$.
\end{remark}

\subsection{Newton polygons}\label{prelim_NP} \mbox{ } \\

Let $X$ denote a $g$-dimensional abelian scheme over an algebraically closed field $\FF$ of positive characteristic $p$.

If $\FF$ is an algebraic closure of $\FF_p$, the finite field of $p$ elements, then there exists a finite subfield $\FF_0\subset \FF$ such that $X$ is isomorphic to the base change to $\FF$ of an abelian scheme $X_0$ over $\FF_0$. Let $W(\FF_0)$ denote the Witt vector ring of $\FF_0$.  Consider the action of Frobenius $\varphi$ on the crystalline cohomology group $H^1_{\rm cris}(X_0/W(\FF_0))$. There exists an integer $n$ such that $\varphi^n$, the composite of $n$ Frobenius actions, is a linear map on $H^1_{\rm cris}(X_0/W(\FF_0))$.
The Newton polygon $\nu(X)$ of $X$ is defined as the multi-set of rational numbers $\lambda$
such that $n\lambda$ are the valuations at $p$ of the eigenvalues of Frobenius for this action. Note that the Newton polygon is independent of the choice of $X_0$, $\FF_0$, and $n$.

Here is an alternative definition, which works for arbitrary algebraically closed field $\FF$.  
For each $n \in \NN$, consider the multiplication-by-$p^n$ morphism $[p^n]:X \to X$ and its kernel $X[p^n]$.
The $p$-divisible group of $X$ is denoted by $X[p^\infty] = \varinjlim X[p^n]$.
For each pair $(c,d)$ of non-negative relatively prime integers, fix
a $p$-divisible group $G_{c,d}$ of codimension $c$, dimension $d$, and thus height $c+d$.
By the Dieudonn\'e--Manin classification \cite{maninthesis}, there is an isogeny of $p$-divisible groups 
\[X[p^\infty] \sim \oplus_{\lambda=\frac{d}{c+d}} G_{c,d}^{m_\lambda},\] where $(c,d)$ ranges over pairs of 
non-negative relatively prime integers.
The Newton polygon is the multi-set of values of the slopes $\lambda$.
By identifying $H^1_{\rm cris}(X/W(\FF))$ with the (contravariant) Dieudonn\'e module of $X$, 
it is possible to show that these definitions are equivalent.

The slopes of the Newton polygon are in $\QQ \cap [0,1]$.
The Newton polygon is typically drawn as a lower convex polygon, 
with endpoints $(0,0)$ and $(2g,g)$ and slopes equal to the values of $\lambda$, with multiplicity $(c+d)m_\lambda$.  
It is symmetric and has integral breakpoints. 
The Newton polygon is an isogeny invariant of $A$; it is determined by the multiplicities $m_\lambda$.

Given an abelian variety or $p$-divisible group $\CA$ defined over a local field of mixed characteristic $(0,p)$, by abuse of notation, we may write $\nu(\CA)$ for the Newton polygon of its special fiber.

In this paper, we use $ord$ to denote the Newton polygon with slopes $0,1$ with multiplicity $1$ and $ss$ to denote the Newton polygon with slope $1/2$ with multiplicity $2$. For $s<t \in \ZZ_{>0}$ with ${\rm gcd}(s,t)=1$, we use $(s/t, (t-s)/t)$ to denote the Newton polygon with slopes $s/t$ and  $(t-s)/t$ with multiplicity $t$. 

\begin{definition}
The \emph{$p$-rank} of $X$ is defined to be $\dim_{\FF_p}\hom(\mu_p,X)$. Equivalently, the $p$-rank of $X$ is the multiplicity of the slope $0$ in the Newton polygon.
\end{definition}

\begin{definition}
Given a finite set of lower convex polygons $\{\nu_i \mid i=1,\dots ,n\}$, $n\geq 2$, each $\nu_i$ having end points $(0,0)$ and $(h_i, d_i)$
and with slope $\lambda$ occurring with multiplicity $m_{i, \lambda}$,
their {\em amalgamate sum} $\sum_{i=1}^n\nu_i$ is the lower convex polygon having end points $(0,0)$ and $(\sum_{i=1}^nh_i, \sum_{i=1}^nd_i)$ 
and with slope $\lambda$ occurring with multiplicity $\sum_{i=1}^n m_{i, \lambda}$.
\end{definition}

For any finite set of $p$-divisible groups $\{G_i \mid i=1,\dots , n\}$, $n\geq 2$, with Newton polygons $\nu(G_i)$,
the Newton polygon of the $p$-divisible group $\oplus _{i=1}^n G_i$ is the amalgamate sum $\sum_{i=1}^n\nu(G_i)$.

\section{Newton polygons of curves with complex multiplication}\label{sec_dim0}

As in Section \ref{prelim_curve}, we fix a monodromy datum $(m,a)$, and consider the $\mu_m$-Galois cover
$C_{(m,a)} \to \PP$ branched at $0,1, \infty$ with local monodromy $a=(a(1),a(2),a(3))$
as in \eqref{Ecurve}. We write $C=C_{(m,a)}$ and let $J=J_{(m,a)}$ be the Jacobian ${\rm Jac}(C)$ of $C$.  The action of $\mu_m$ on $C$ induces an action of $\QM$ on $J$.
Also, the equation of $C$ naturally defines integral models $\cC$ and $\CJ={\rm Jac}(\cC)$ of $C$ and $J$ over $\ZZ$ \cite[Section 4]{wewersthesis};
the curve $\cC$ has good reduction at all primes $p$ such that $p\nmid m$ by \cite[XIII, Corollary~2.12, Proposition 5.2]{SGA1}.

\subsection{Shimura--Taniyama method}\label{sec_ST} \mbox{ } \\

It is well-known that $J$ is an abelian variety with complex multiplication, but we record a proof here with a refined statement on the CM algebra contained in $\End(J_{\QQ^{\rm alg}})\otimes \QQ$.

We say that an abelian variety $A$ over $\QQ^{\rm alg}$ has complex multiplication (CM) by a $\QQ$-algebra $E$ if $E$ is an \'etale $\QQ$-subalgebra of $\End(A_{\QQ^{\rm alg}})\otimes\QQ$ of degree $2\dim A$ over $\QQ$. 
In particular, if $E=\prod E_i$ then $A$ has CM by $E$ if and only if $A$ is isogenous to $\prod A_i$ with $A_i$ an abelian variety with CM by $E_i$. Also, if $A$ has CM by $E$ then $H_1(A,\QQ)$ is free of rank $1$ over $E$ (\cite[Definition 3.2, and Proposition 3.6]{milneCM}).

\begin{lemma}\label{lem_CM}
The abelian variety $J$ has complex multiplication by $\prod_{d} K_d$, where the product is taken over all $d$ such that 
$1<d\mid m$ and $d\nmid a(i)$ for any $i=1,2,3$. 
\end{lemma}

\begin{proof}
By Hodge theory, $J_{\overline{\QQ}}$ is isogenous to $\prod_{1<d\mid m}A_d$, where $A_d$ is an abelian variety whose first Betti cohomology group is isomorphic to $\oplus_{\tau\text{ of order }d}V_\tau$.

For $0\neq n\in \ZZ/m\ZZ$, the order $d$ of $n$ is $d=m/\gcd(n,m)$. 
Let $x_n$ be the number of elements in $a$ which are not divisible by $d$. 
If none of the $a(i)$ is divisible by $d$, then $x_n=3$. Otherwise, $x_n=2$ since $\gcd(m, a(1),a(2),a(3))=1$
and $a(1)+a(2)+a(3) \equiv 0 \bmod m$.
For example, when $\gcd(n,m)=1$, then $d=m$ and hence $x_n=3$. 
For example, if $m$ is even and $n=m/2$, then $d=2$ and $x_n=2$. 
By \eqref{DMeqn}, $\frf(\tau_n)+\frf(\tau_{-n})=x_n-2$.
Hence $A_d=\{1\}$, the trivial abelian variety if $x_n=2$, and $A_d$ is a simple abelian variety of dimension $[K_d:\QQ]/2$ such that $\End(A_d)\otimes \QQ=K_d$ otherwise.

Then $J$ has complex multiplication by the product of the fields $K_d$ where the product is taken over all $d$ such that if $x_n=3$.
\end{proof}

By Lemma \ref{lem_CM}, the abelian variety $J$ has potentially good reduction everywhere. 
This means that there exists an abelian scheme over some finite extension of $\ZZ_p$ such that its generic fiber is $J$.
For $p \nmid m$, since $C$ already has good reduction at $p$, so does $J$; no extension of $\ZZ_p$ is needed, the Jacobian $\CJ$ is a smooth integral model of $J$ defined over $\ZZ_p$.
The $\QM$-action on $J$ extends naturally to a $\QM$-action on $\CJ$.  

Let $\CJ[p^\infty]$ be the associated $p$-divisible group scheme of $\CJ$. The $\QM$-action on $\CJ$ induces a $(\QM\otimes_\QQ \QQ_p$)-action on $\CJ[p^\infty]$ and thus a canonical decomposition
\[\CJ[p^\infty]=\oplus_{\co \in \CO'} \CJ[\p_\co^\infty],\]
where $\CO'$ is the subset of $\CO$ with $d_\co \nmid a(i)$ for any $i=1,2,3$ and each $p$-divisible group $\CJ[\p_\co^\infty]$ has height $\# \co$. 

To state the theorem by Shimura--Taniyama and Tate on $\nu(\CJ[\p_\co^\infty])$, we introduce the following notation. 
Recall that $\frf(\tau)=\dim_{\CC} V^+_\tau$. 
By the proof of Lemma \ref{lem_CM}, $\frf(\tau)\in \{0,1\}$ and $\frf(\tau)+\frf(\tau^*)=1$ for all $\tau \in \CT$ such that the order of $\tau$ does not divide $a(i)$ for any $i=1,2,3$. For $\epsilon \in \{0,1\}$, define
\begin{equation} \label{EdefS}
S_\epsilon=\{\tau \in \CT \mid \text{the order of }\tau \text{ does not divide }a(1),a(2), a(3), \text{ and }\frf(\tau)=\epsilon\}.
\end{equation}
For $\co \in \CO'$, set $\alpha_\co=\#( \co\cap S_1)$ and $\beta_\co=\#(\co\cap S_0)$. Note that $\alpha_\co+\beta_\co=\# \co$.

\begin{theorem} \label{thm_ST} (Shimura--Taniyama formula \cite[Section 5]{Tate})
The only slope of the Newton polygon $\nu(\CJ[\p_\co^\infty])$ is $\alpha_\co/\# \co$.
\end{theorem}

\begin{proof}
For completeness, we briefly sketch Tate's local proof as in \cite[Section 5]{Tate}. First, we recall the notion of a 
$p$-divisible group with complex multiplication. 

Let $G$ be a $p$-divisible group defined over the ring of integers ${\mathcal O}_L$ of
a finite extension $L$ of $\Q_p$ such that $L\subset \Q_p^{\rm alg}$. 
We say that  $G$ over ${\mathcal O}_L$ has complex multiplication by a local field $K$, for $K$ a finite extension of $\Q_p$, 
if $G$ has height $[K:\Q_p]$ and is equipped with a $\Q_p$-linear action of $K$ defined over ${\mathcal O}_L$ such that, for each $\tau\in H:={\rm Hom}_{\Q_p}(K, \Q_p^{\rm alg})$, 
$$\cf(\tau):=\dim_{\Q_p^{\rm alg}} ({\rm Lie}(G)\otimes_{{\mathcal O}_L} \Q_p^{\rm alg})_\tau\in\{0,1\}.$$ 
For $\Phi:=\{\tau\in  H \mid \cf(\tau)=1\}$,
the pair $(K,\Phi)$ is called the CM-type of $G$. 

Let $k_L$ denote the residue field of $L$ and set $G_0:=G\times_{{\mathcal O}_L} k_L$, the reduction of $G$ over $k_L$. We observe that, by definition,  if $G$ over ${\mathcal O}_L$ has CM-type $(K,\Phi)$, then $G_0$ is isoclinic  (i.e., the Newton polygon of $G_0$ has only one slope) of slope $\#\Phi/[K:\Q_p]$. 

Indeed,  the existence of a $\Q_p$-linear embedding of $K$ into ${\rm End}(G)\otimes \QQ$, with $[K:\Q_p]={\rm height}(G)$, implies that $G_0$ is isoclinic.
Also, by the definition of CM-type, 
the dimension of $G$
is equal to $\#\Phi$, because
$$\dim (G)=  {\rm rk}_{{\mathcal O}_L}  {\rm Lie}(G)=
 \sum_{\tau\in H}\cf(\tau)=\#\Phi.$$  We deduce that the slope of $G_0$ is  $\frac{\dim (G)}{{\rm height }(G)}=\frac{\#\Phi}{[K:\Q_p]}$.

To conclude, it suffices to observe that for $p \nmid m$, Lemma \ref{lem_CM} implies that, for each $\co\in\CO'$, the $p$-divisible group $\CJ[\p_\co]$, after passing to a finite extension of $\ZZ_p$,\footnote{The Newton polygon is independent of the definition field of $\CJ$ and we pass to a finite extension of $\ZZ_p$ such that the CM-action is defined over this larger local ring so that we can apply Tate's theory.} has complex multiplication by $K_{d_\co,\p_\co}$ with CM-type $(K_{d_\co,\p_\co},\co\cap S_1)$.
\end{proof}

\begin{corollary}\label{ST_ss}
Assume that all orbits $\co\in \CO'$ are self-dual, i.e., $\co=\co^*$. Then $\CJ$ has supersingular reduction at $p$.
\end{corollary}

\begin{proof}
For each $\co\in\CO'$, if $\tau\in \co$, then $\tau^*\in \co$. Hence $\alpha_\co=\beta_\co$, and the only slope of the Newton polygon $\nu(\CJ[\p_\co^\infty])$ is $\alpha_\co/(\alpha_\co+\beta_\co) = 1/2$. 
\end{proof}

\begin{remark} \label{pinert}
Let $K_m^+$ be the maximal totally real subfield of $\Km$. If each (or, equivalently, one) prime of $K_m^+$ above $p$ is inert in $\Km/K_m^+$, then all $\sigma$-orbits $\co\in\CO'$ are self-dual. E.g., if $p\equiv -1\bmod m$, then  for all $n\in(\ZZ/m\ZZ)^*$, the associated orbit is $\co_n=n\langle p\rangle=\{n,-n\}=\co_{-n}=\co^*_n$.
\end{remark}

\begin{example} \label{example_ss}
Let $g \geq 1$, $m=2g+1$ and $a = (1,1,m-2)$. 
The equation $y^m=x(x-1)$ 
defines a smooth projective curve $\cC$ over $\ZZ[1/m]$ with geometrically irreducible fibers. 
It has genus $g$ by \eqref{Egenus}.
Its Jacobian $\CJ$ over $\ZZ[1/m]$ has complex multiplication by $\prod_{1<d\mid m}\Kd$. 
Suppose that $p\nmid m$ and the order of $p\in (\ZZ/m\ZZ)^*$ is even 
(i.e., the inertia degree of $(\Km)_{\p}$ is even for any $\p$ over $p$).
When $m$ is prime, this implies that all primes of $K_m^+$ above $p$ are inert in $K_m/K_m^+$.
Then $\CJ$ has supersingular reduction at $p$ by Corollary \ref{ST_ss}. 
When $p\nmid 2m$, this is \cite[Theorem 1]{honda}. See also \cite[Lemma 1.1]{GR}.
\end{example}

\begin{remark}
The Newton polygon of the Fermat curve $F_m: x^m+y^m=z^m$ is studied in \cite{yui}, with the connection 
to Jacobi sums going back to \cite{Weil}.
Fix $k$ with $2 \leq k \leq m-1$ and consider the inertia type $a=(1, k-1, m-k)$.
Then the $\mu_m$-Galois cover $\cC_{m,a} \to \PP$ with inertia type $a$ is a quotient of the Fermat curve.  
In certain cases, this is sufficient to determine the Newton polygon of $\cC_{m,a}$.
Let $f$ be the order of $p$ modulo $m$.
If $f$ is even and $p^{f/2} \equiv -1 \bmod m$, then $F_m$ is 
supersingular, so $\cC_{m,a}$ is supersingular as well.
If $f$ is odd, then the slope $1/2$ does not occur in the Newton polygon of $F_m$ or $\cC_{m,a}$.
Information about the $p$-rank of $F_m$ and $\cC_{m,a}$ can be found in \cite{gonzalez}.
\end{remark}

\subsection{Slopes with large denominators} \mbox{ } \\

In this section, we use Theorem \ref{thm_ST} to construct curves of genus $g \geq 1$ whose 
Newton polygon contains only slopes with large denominators. 

Consider a monodromy datum $(m,a)$ with $m=2g+1$ and $a=(a(1),a(2),a(3))$. 

In this section, we assume that $d \nmid a(i)$ for any $1<d|m$.

For convenience, via the identification $\CT=\ZZ/m\ZZ$, we identify the sets $S_0$ and $S_1$ from \eqref{EdefS} as subsets of $\ZZ/m\ZZ$.

If $p$ is a rational prime, not dividing $m$, 
we write $\langle p\rangle\subset (\mathbb{Z}/m\ZZ)^*$ for the cyclic subgroup generated by the congruence class of $p\bmod m$. Then $\langle p \rangle$ acts naturally on $\mathbb{Z}/m\ZZ$. Let $n\langle p \rangle$ denote the $\langle p \rangle$-orbits in $\mathbb{Z}/m\ZZ$ where $n\not\equiv 0$.


The following proposition is a special case of Theorem \ref{thm_ST}.

\begin{proposition}\label{prop_ST}
Assume $p \nmid m$.
Then the slopes of the Newton polygon $\nu(\CJ)$ at $p$ are naturally indexed by the cosets of 
$\hp$ in $(\mathbb{Z}/d\ZZ)^*$ for all $1<d\mid m$.
For each orbit $n\hp$, 
the associated slope is \[\lambda_{n\hp}:=\frac{\# n\hp \cap S_1}{\#n\hp}.\]
\end{proposition}

Note that when the inertia type $a$ is fixed, the Newton polygon $\nu(\CJ)$ at a prime $p$ depends only on the associated subgroup $\hp\subset (\mathbb{Z}/m\ZZ)^*$. 
In particular, if $m=2g+1$ is prime, then $\nu(\CJ)$ at $p$ depends only on the order of $p$ in $ (\mathbb{Z}/m\ZZ)^*$. See also \cite[Theorem~2]{honda} for the case when $a=(1,1,m-2)$.

\begin{corollary}   \label{Clargeden}
Assume that $m=2g+1$ is prime and let $f$ be a prime divisor of $g$.
If $p$ is a prime with $p \nmid m$ such that the reduction of $p$ has order $f$ in $ (\mathbb{Z}/m\ZZ)^*$,
then 
every slope $\lambda$ of the Newton polygon $\nu(\CJ)$ at $p$ with $\lambda\neq 0,1$ has denominator $f$.  
In particular, if the $p$-rank of $\CJ$ is $0$, then every slope has denominator $f$.
\end{corollary}

For any $m$, $g$ and $f$ as above, there are infinitely many primes $p$ satisfying the hypotheses of Corollary \ref{Clargeden} by the Chebotarev density theorem.
\begin{proof}
Under the hypotheses, if $\lambda_{n\hp}\neq 0,1$, then $f$ is the denominator of the fraction $\lambda_{n\hp}$ by Proposition \ref{prop_ST}. When the $p$-rank is $0$, then there is no slope $0$ or $1$ by definition.
\end{proof}

\begin{example} \label{Enotsg7}
Let $g=14$ and $m=29$. 
For $p \equiv 7,16,20,23,24,25 \bmod 29$, the inertia degree is $f=7$. 
For each choice of the inertia type, the Newton polygon is $(2/7, 5/7)\oplus (3/7, 4/7)$.
\end{example}

A prime number $\ell$ is a \emph{Sophie Germain prime} if $2\ell+1$ is also prime. The rest of the section focuses on the case when $g$ is a Sophie Germain prime.

\begin{corollary} 
Suppose $g$ is an odd Sophie Germain prime. Let $p$ be a prime, $p \neq 2g+1$.

Then, one of the following occurs.
\begin{enumerate}
\item 
If $p \equiv 1 \bmod 2g+1$, then $\nu(\CJ)=ord^g$.
\item 
If $p \equiv -1 \bmod 2g+1$, then $\nu(\CJ)=ss^g$.
\item If $p$ has order $g$ modulo $2g+1$, then $\nu(\CJ)=(\alpha/g,(g-\alpha)/g)$, 
for $\alpha=\# \hp \cap S_1$.
\item If $p$ has order $2g$ modulo $2g+1$, then $\nu(\CJ)=ss^g$.
\end{enumerate}
\end{corollary}

\begin{proof}
Let $m=2g+1$.
Under the Sophie Germaine assumption on $g$, a prime $p\neq m$ has order either $1$, $2$, $g$, or $2g$ modulo $m$. 
Cases (2) and (4) follow from Corollary \ref{ST_ss} and Cases (1) and (3) follow from Proposition \ref{prop_ST}.
\end{proof}

Note that $p$ has order $g$ modulo $2g+1$ if and only if $p$ is a quadratic residue other than $1$ 
modulo $2g+1$. In this case, if $a=(1,1,2g-1)$, then $S_1=\{1, \ldots, g\}$ and 
$\alpha$ is the number of quadratic residues modulo $2g+1$ in $S_1$. 

\begin{example} \label{Esgg5}
 Let $g=5$ and $m=11$.  Let $p \equiv 3,4,5,9 \bmod 11$.
If $a=(1,1,9)$, then $\nu(\CJ)=(1/5,4/5)$.
If $a=(1,2, 8)$, then $\nu(\CJ)=(2/5,3/5)$. 
\end{example}

\begin{example} \label{Esgg11}
Let $g=11$ and $m=23$.  
Let $p \equiv  2, 3, 4, 6, 8, 9, 12, 13, 16, 18 \bmod 23$.
If $a=(1,1,21)$, then $\nu(\CJ)=(4/11,7/11)$.
If $a=(1,4,18)$, then $\nu(\CJ)=(1/11,10/11)$.
\end{example}

In Examples \ref{Esgg5} - \ref{Esgg11}, the listed Newton polygons are the only ones that 
can occur, under these conditions on $p$, as $a$ varies among all possible inertia types for $\mu_m$-Galois covers
of the projective line branched at three points. 
We did not find other examples of curves whose Newton polygon has slopes $1/d$ and $(d-1)/d$ using this method.

\begin{example} \label{Esgg1013}
	Let $g=1013$ and $m=2027$.  
	Suppose the congruence class of $p$ modulo $2027$ is contained in $\langle 3 \rangle$ in $(\ZZ/2027\ZZ)^*$.
	If the inertia type is $a=(1,1,2025)$, then $\nu(\CJ)=(523/1013,490/1013)$.
\end{example}

\section{Tables}\label{sec_table}
The following tables contain all the Newton polygons which occur for cyclic degree $m$ covers of the projective line branched at 3 points
when $3 \leq m \leq 12$.
Each inertia type $a=(a_1, a_2, a_3)$ is included, up to permutation and the action of $(\ZZ/m)^*$.

The signature is computed by \eqref{DMeqn} and written as $(f(1), \ldots, f(m-1))$. We denote by $ord$ the Newton polygon of $G_{0,1} \oplus G_{1,0}$ 
which has slopes $0$ and $1$ with multiplicity $1$
and by $ss$ the Newton polygon of $G_{1,1}$ which has slope $1/2$ with multiplicity $2$.

\begin{center}
	$m=3$\\
	\begin{tabular}{  |c|c|c|c|  }
		\hline
		& p & $ 1 \bmod 3$ & $2 \bmod 3$ \\
		\hline
		& {prime orbits} & split & $(1,2)$  \\
		\hline
		a & signature & &\\
		\hline
		$(1,1,1)$  &  $(1,0)$ &  $ord$ &  $ss$  \\
		\hline
	\end{tabular}
\end{center}

\bigskip

\begin{center}
	$m=4$\\
	\begin{tabular}{  |c|c|c|c|  }
		\hline
		& p & $ 1 \bmod 4$ & $3 \bmod 4$ \\
		\hline
		& {prime orbits} & split & $(1,3),(2)$  \\
		\hline
		a & signature & &\\
		\hline
		$(1,1,2)$  &  $(1,0,0)$ &  $ord$ &  $ss$  \\
		\hline
	\end{tabular}
\end{center}

\bigskip

\begin{center}
	$m=5$\\
	\begin{tabular}{  |c|c|c|c|c|  }
		\hline
		& p & $ 1 \bmod 5$ & $2,3 \bmod 5$ & $ 4 \bmod 5$\\
		\hline
		& {prime orbits} & split & $(1,2,3,4)$ & $(1,4)$, $(2,3)$  \\
		\hline
		a & signature & && \\
		\hline
		$(1,1,3)$  &  $(1,1,0,0)$ &  $ord^2$ &  $ss^2$  & $ss^2$\\
		\hline
	\end{tabular}
	
\end{center}

\bigskip

\begin{center}
	$m=6$\\
	\begin{tabular}{  |c|c|c|c|  }
		\hline
		& p & $1 \bmod 6$ & $ 5 \bmod 6$\\
		\hline
		& {prime orbits} & split & $(1,5),(2,4),(3)$ \\
		\hline
		a & signature & & \\
		\hline
		$(1,1,4)$  &  $(1,1,0,0,0)$ &  $ord^2$ & $ ss^2$  \\
		\hline
		$(1,2,3)$  &  $(1,0,0,0,0)$ &  $ord$ & $ ss$  \\
		\hline
	\end{tabular}
\end{center}

\bigskip

\begin{small}
	
	\begin{center}
		$m=7$\\
		\begin{tabular}{  |c|c|c|c|c|c|  }
			\hline
			& p & $1 \bmod 7$ & $2,4 \bmod 7$& $3,5 \bmod 7$ & $6 \bmod 7$\\
			\hline
			& {prime orbits} & split & $(1,2,4),(3,5,6)$ & $(1,2,3,4,5,6)$ & $(1,6),(2,5),(3,4)$  \\
			\hline
			a & signature & &&& \\
			\hline
			(1,1,5)  &  (1,1,1,0,0,0) &  $ord^3 $& (1/3,2/3)& $ss^3 $ & $ss^3$  \\
			\hline
			(1,2,4)  &  (1,1,0,1,0,0) &  $ord^3$ &  $ord^3$& $ss^3$ & $ss^3$  \\
			\hline
		\end{tabular}
	\end{center}
\end{small}

\bigskip

\begin{small}
	
	\begin{center}
		$m=8$\\
		\begin{tabular}{  |c|c|c|c|c|c|  }
			\hline
			& p & $1 \bmod 8$ & $ 3 \bmod 8$& $ 5\bmod 8$ & $ 7 \bmod 8$\\
			\hline
			&  &   & $(1,3),(2,6)$ & $(1,5),(3,7)$ & $(1,7),(2,6)$  \\
			&{prime orbits}& split & $(5,7),(4)$ & $(2),(4),(6)$ & $(3,5),(4)$  \\
			\hline
			a & signature & &&& \\
			\hline
			(1,1,6)  &  (1,1,1,0,0,0,0) &  $ord^3 $& $ord^2 \oplus ss$ & $ord \oplus ss^2$ & $ss^3$  \\
			\hline
			(1,2,5)  &  (1,1,0,0,1,0,0) &  $ord^3$ &  $ss^3$& $ord^3$ & $ss^3$  \\
			\hline
			(1,3,4)  &  (1,0,1,0,0,0,0) & $ ord^2$ & $ ord^2$& $ss^2$ & $ss^2$  \\
			\hline
		\end{tabular}
	\end{center}
\end{small}

\bigskip

\begin{center}
	
	\begin{small}
		$m=9$\\
		\begin{tabular}{  |c|c|c|c|c|c|  }
			\hline
			& p & $ 1 \bmod 9$ & $ 2,5 \bmod 9$& $4,7 \bmod 9$ & $8 \bmod 9$\\
			\hline
			&  &   & $(1,2,4,8,7,5)$ & $(1,4,7),(2,8,5)$ & $(1,8),(2,7)$  \\
			&{prime orbits}& split & $(3,6)$ & $(3),(6)$ & $(4,5),(3,6)$  \\
			\hline
			a & signature & &&& \\
			\hline
			(1,1,7)  &  (1,1,1,1,0,0,0,0) &  $ord^4$ & $ss^4$& $(1/3,2/3) \oplus ord$ & $ss^4$  \\
			\hline
			(1,2,6)  &  (1,1,0,0,1,0,0,0) &  $ord^3$ &  $ss^3$& $(1/3,2/3)$ & $ss^3$  \\
			\hline
			(1,3,5)  &  (1,1,0,1,0,0,0,0) &  $ord^3$ &  $ss^3$& $(1/3,2/3)$ & $ss^3 $ \\
			\hline

		\end{tabular}
	\end{small}
\end{center}

\bigskip

\begin{center}
	$m=10$\\
	\begin{tabular}{  |c|c|c|c|c|  }
		\hline
		& p & $ 1 \bmod 10$ & $3,7 \bmod 10$ & $ 9 \bmod 10$\\
		\hline
		& & & $(1,3,9,7)$ & $(1,9),(2,8)$   \\
		&{prime orbits} & split & $(2,6,8,4),(5)$ & $(3,7),(4,6),(5)$  \\
		\hline
		a & signature & && \\
		\hline
		(1,1,8)  &  (1,1,1,1,0,0,0,0,0) &  $ord^4$ &  $ss^4$  & $ss^4$\\
		\hline
		(1,2,7)  &  (1,1,1,0,0,1,0,0,0) &  $ord^4$ &  $ss^4$  & $ss^4$\\
		\hline
		(1,4,5)  &  (1,0,1,0,0,0,0,0,0) &  $ord^2$ &  $ss^2$  & $ss^2$\\
		\hline
	\end{tabular}
\end{center}

\bigskip

\begin{small}
	
	\begin{center}
		$m=11$\\
		\begin{tabular}{  |c|c|c|c|c|c|  }
			\hline
			& p & $ 1 \bmod 11$ & $2,6,7,8 \bmod 11$ & $ 3,4,5,9 \bmod 11$ & $ 10 \bmod 11$\\
			\hline
			& & &  & $(1,3,4,5,9)$ & $(1,10),(2,9)$   \\
			&{prime orbits} & split & inert & $(2,6,7,8,10)$ & $(3,8),(4,7),(5,6)$  \\
			\hline
			a & signature & &&& \\
			\hline
			(1,1,9)  &  (1,1,1,1,1,0,0,0,0,0) &  $ord^5$ &  $ss^5$  & $(1/5,4/5)$& $ss^5$\\
			\hline
			(1,2,8)  &  (1,1,1,0,0,1,1,0,0,0) &  $ord^5$ &  $ss^5$  & $(2/5,3/5)$& $ss^5$\\
			\hline
		\end{tabular}
	\end{center}
\end{small}

\bigskip

\begin{center}
	
	\begin{small}
		
		$m=12$\\
		\begin{tabular}{  |c|c|c|c|c|c|  }
			\hline
			& p & $1 \bmod 12$ & $5 \bmod 12$ & $ 7 \bmod 12$&  $ 11 \bmod 12$\\
			\hline
			& &   &  &$(1,7),(3,9)$ & $(1,11),(4,8)$  \\
			& & & $(1,5),(2,10),(3)$ & $(2),(4),(5,11)$ & $(3,9),(2,10)$  \\
			& {prime orbits}& split & $(4,8),(6),(7,11),(9)$ & $(6),(8),(10)$ & $(5,7),(6)$ \\
			\hline
			a & signature & &&& \\
			\hline
			(1,1,10)  &  (1,1,1,1,1,0,0,0,0,0,0) &  $ord^5$ &  $ord^3 \oplus ss^2$  & $ord^2 \oplus ss^3$ & $ss^5$\\
			\hline
			(1,2,9)  & (1,1,1,0,0,0,1,0,0,0,0) &  $ord^4$ &  $ord \oplus ss^3$  & $ord^3 \oplus ss$ & $ss^4$\\
			\hline
			(1,3,8)  &  (1,1,0,0,1,0,0,0,0,0,0) &  $ord^3$ &  $ord^2 \oplus ss$  & $ord \oplus ss^2$ & $ss^3$\\
			\hline
			(1,4,7)  &  (1,1,0,1,0,0,1,0,0,0,0) &  $ord^4$ &  $ss^4$  &$ord^4$ & $ss^4$\\
			\hline
			(1,5,6)  & (1,0,1,0,1,0,0,0,0,0,0) &  $ord^3$ &  $ord^3$  & $ss^3$ & $ss^3$\\
			\hline
		\end{tabular}
	\end{small}
\end{center}

\section{Applications} \label{Sapplication}

In the previous section, we computed the Newton polygons of cyclic degree $m$ covers of the projective line branched at $3$ points.
We carried out the calculation of the Newton polygon for all inertia types that arise when  $m \leq 23$.
Many of these were not previously known to occur for the Jacobian of a smooth curve.
We collect a list of the most interesting of these Newton polygons, restricting to the 
ones with $p$-rank $0$ and $4 \leq g \leq 11$.
In the third part of the section, we deduce some results for arbitrarily large genera $g$.

By \cite[Theorem 2.1]{VdGVdV}, if $p=2$ and $g \in \NN$, then there exists a supersingular curve of genus $g$ defined over 
$\overline{\FF}_2$ (or even over $\FF_2$).  For this reason, we restrict to the case that $p$ is odd in the following result.

\begin{theorem} \label{Tapp1} (Theorem \ref{Tintro1})
Let $p$ be odd.
There exists a smooth supersingular curve of genus $g$ defined over $\overline{\FF}_p$ in the following cases:
\begin{center}
\begin{tabular}{  |c|c|c|  }
\hline
		genus & congruence & where \\
\hline
4 & $p \equiv 2 \bmod 3$ & m=9, $a=(1,1,7)$\\
& $p \equiv 2,3,4 \bmod 5$ & $m=10$, $a=(1,1,8)$ \\
\hline
5 & $p \equiv 2,6,7,8,10 \bmod 11$ & $m=11$, any $a$\\
\hline 
6 & $p \not\equiv 1,3,9 \bmod 13$ & $m=13$, any $a$\\
& $p \equiv 3,5,6 \bmod 7$ & $m=14$, $a=(1,1,12)$\\
\hline
7 & $p \equiv 14 \bmod 15$ & $m=15$, $a=(1,1,13)$\\
& $p \equiv 15 \bmod 16$ & $m=16$, $a=(1,1,14)$\\
\hline
8  &  $p \not \equiv 1 \bmod 17$   & $m=17$, any $a$\\
		\hline	
9  & $p \equiv 2,3,8,10,12,13,14,15,18 \bmod 19$ & $m=19$, any $a$ \\
		\hline
10 & $p \equiv 5,17,20 \bmod 21$ & $m=21$, $a=(1,1,19)$\\
\hline
11 & $p \equiv 5,7,10,11,14,15,17,19,20,21,22 \bmod 23$& $m=23$, any $a$\\
\hline	
\end{tabular}
\end{center}
\end{theorem}

\begin{proof}
We compute the table using Corollary \ref{ST_ss}, and Remark \ref{pinert}. 
The genus is determined from \eqref{Egenus}.

For example, for $g=5$,  the congruence classes of $p\bmod 11$  are the quadratic non-residues modulo $11$. 
A prime above $p$ is inert in $K_{11}/K_{11}^+$  if and only if $p$ is a quadratic non-residue modulo $11$.
(The same holds for $g=9$ and $m=19$, 
and also $g=11$ and $m=23$).

For example, for $g=4$, 
the condition $p \equiv 3,7,9 \bmod 10$ covers all the cases $p \equiv 2,3,4 \bmod 5$ since $p$ is odd. For $p\equiv 3,7\bmod 10$, $p$ is inert in $K_5$.
For $p\equiv -1\bmod 10$, each orbit $\co\in\CO'$ is self-dual.
\end{proof}

\begin{remark} \label{Runlikely}
The existence of a smooth supersingular curve of genus $9$, $10$ or $11$ is especially interesting for the following reason. 
The dimension of $\CA_g$ is $(g+1)g/2$ and the dimension of the supersingular locus in $\CA_g$ is $\lfloor g^2/4 \rfloor$.
Thus the supersingular locus has codimension $25$ in $\CA_9$, $30$ in $\CA_{10}$ and $36$ in $\CA_{11}$.
The dimension of $\CM_g$ is $3g-3$ for $g \geq 2$.
Since $\dim(\CM_9) = 24$, $\dim(\CM_{10})=27$, and $\dim(\CM_{11})=30$ the supersingular locus and open Torelli locus form an unlikely intersection in $\CA_9$, $\CA_{10}$
and $\CA_{11}$.
See \cite[Section 5.3]{priesCurrent} for more explanation.
\end{remark}

\begin{remark}
In future work, when $5 \leq g \leq 9$, we prove there exists a smooth supersingular curve of genus $g$ defined over $\overline{\FF}_p$ 
for sufficiently large $p$ satisfying other congruence conditions.
\end{remark}

In the next result, we collect some other new examples of Newton polygons of smooth curves with $p$-rank $0$.

\begin{theorem} \label{TotherNP}
There exists a smooth curve of genus $g$ defined over $\overline{\FF}_p$ with the given Newton polygon
of $p$-rank $0$ in the following cases:
\begin{center}
	\begin{tabular}{  |c|c|c|c|  }
		\hline
		genus &	Newton polygon & congruence & where \\
		\hline
		5  & $(1/5,4/5)$  & $3,4,5,9 \bmod 11$  & $m=11$, $a=(1,1,9)$ \\
		\hline
		5  & $(2/5,3/5)$  & $3,4,5,9 \bmod 11$  & $m=11$, $a=(1,2,8)$ \\
		\hline
		6  & $(1/3,2/3)^2$  & $3,9 \bmod 13$ & $m=13$, $a=(1,2,10)$\\
		  &   & $9,11 \bmod 14$ & $m=13$, $a=(1,1,12)$\\
		\hline
		7  & $(1/4,3/4) \oplus ss^3$  & $2,8 \bmod 15$  & $m=15$, $a=(1,1,13)$   \\
		\hline
		9  & $(4/9,5/9)$  & $4,5,6,9,16,17 \bmod 19$ & $m=19$, $a=(1,2,16)$   \\
		\hline
		9  & $(1/3,2/3)^3$  & $4,5,6,9,16,17 \bmod 19$  & $m=19$, $a=(1,1,17)$   \\
		   & & $7,11 \bmod 19$  & $m=19$, $a=(1,2,16)$   \\
		\hline
		10 & $(1/3,2/3)^3 \oplus ss$ & $p \equiv 2 \bmod 21$ & $m=21$, $a=(1,1,19)$\\
		\hline
		11 & $(1/11, 10/11)$ & $2,3,4,6,8,9,12,13,16,18 \bmod 23$ & $m=23$, $a=(1,4,18)$ \\
		\hline
		11 & $(4/11, 7/11)$ & $2,3,4,6,8,9,12,13,16,18 \bmod 23$ & $m=23$, $a=(1,1,21)$ \\
		\hline
			
	\end{tabular}
\end{center}

\end{theorem}

\begin{proof}
We compute the table using the Shimura--Tanayama method, as stated in Proposition \ref{prop_ST}.
The genus is determined from \eqref{Egenus}, and the signature type from \eqref{DMeqn}.

For example, for $m=15$ and $a=(1,1,13)$, the curve $C_{(m,a)}$ has genus 7 and signature type 
$(1,1,\dots ,1,0,0,\dots, 0)$. Let $p$ be a prime such that $p\equiv 2\bmod 15$.
The congruence class of $p$ has order $2$ modulo $3$ and order $4$ modulo $5$. 
Hence the prime $p$ is inert in $K_3$ and in $K_5$, and splits as a product of two primes in $K_{15}$. 
We write $\p_3$ (resp.\  $\p_5$, and $\p_{15},{\p}'_{15}$) for the primes of $K_3$ (resp.\ $K_5$, and $K_{15}$) above $p$.

Continuing this case, by Corollary \ref{ST_ss}, the Newton polygon of $\CJ[\p^\infty_3]$ (resp.\ $ \CJ[\p^\infty_5]) $ has slope 1/2, with multiplicity 1 (resp.\ 2). 
On the other hand, the two orbits in $(\ZZ/15\ZZ)^*$ are $\co=\langle 2\rangle=\{2,4,8,1\}$
and $\co'=7\langle 2\rangle=\{7,14,13,11\}$. (In particular, $\co'=\co^*$, hence $\p'_{15}={\p}^*_{15}$ in $K_{15}$.)
Hence, $\alpha_\co= 3$ and $\alpha_{\co'}=1$.
By Theorem \ref{thm_ST}, the Newton polygon of $\CJ[\p^\infty_{15}]$ (resp.\ $\CJ[{\p'}^\infty_{15}]$) has slope $3/4$ (resp.\ $1/4$).

For example, for $m=21$ and $a=(1,1,19)$, the curve $C_{(m,a)}$ has genus 10 and signature type 
$(1,1,\dots,1,0,0,\dots,0)$. The congruence class of $p \equiv 2 \bmod 21$ has order $2$ modulo $3$  and order $3$ modulo $7$.
Hence the prime $p$ is inert in $K_3$, and splits as a product of two primes in $K_7$ and in $K_{21}$. 
We write $\p_3$ (resp.\  $\p_7,\p'_7$, and $\p_{21},{\p}'_{21}$) for the primes of $K_3$ (resp.\ $K_7$, and $K_{21}$) above $p$. 

Continuing this case, by Corollary \ref{ST_ss}, the Newton polygon of $\CJ[\p^\infty_3]$ has slope 1/2, with multiplicity 1. 
The two orbits in $(\ZZ/7\ZZ)^*$ are $\co_7=\langle 2\rangle=\{2,4,1\}$
and $\co_7'=3\langle 2\rangle=\{3,6,5\}$.
Hence, $a_{\co_7}= 2$ and $a_{\co_7'}=1$.
By Theorem \ref{thm_ST}, the Newton polygon of $\CJ[\p^\infty_{7}]$ (resp.\ $\CJ[{\p'}^\infty_{7}]$) has slope $2/3$ (resp.\ $1/3$).
Similarly, the two orbits in $(\ZZ/21\ZZ)^*$ are $\co_{21}=\langle 2\rangle=\{2,4,8,16,11,1\}$
and $\co_{21}'=5\langle 2\rangle=\{5,10, 20,19,17\}$.  Again $a_{\co_{21}}= 4$ and $a_{\co_{21}'}=2$. 
Thus the Newton polygon of $\CJ[\p^\infty_{21}]$ (resp.\ $\CJ[{\p'}^\infty_{21}]$) has slope $4/6=2/3$ (resp.\ $2/6=1/3$).
\end{proof}

Consider the $p$-divisible group $G_{1, d-1} \oplus G_{d-1,1}$ with slopes $1/d, (d-1)/d$.

\begin{theorem} \label{Tapp2} (Theorem \ref{Tintro2})
For the following values of $d$ and $p$ and for all $g \geq d$, 
there exists a smooth curve of genus $g$ defined over $\overline{\FF}_p$
whose Jacobian has $p$-divisible group isogenous
to $(G_{1, d-1} \oplus G_{d-1,1}) \oplus (G_{0,1} \oplus G_{1,0})^{g-d}$:
\begin{enumerate}
\item $d=5$ for all $p \equiv 3,4,5,9 \bmod 11$;
\item $d=11$ for all $p \equiv 2,3,4,6,8,9,12,13,16,18 \bmod 23$.
\end{enumerate}
\end{theorem}

\begin{proof}
By Example \ref{Esgg5} (resp.\ Example \ref{Esgg11})
for $d=5$ (resp.\ $d=11$), under this congruence condition on $p$, 
there exists a smooth projective curve of genus $g=d$ 
defined over $\overline{\FF}_p$ whose $p$-divisible group is isogenous to $G_{1, d-1} \oplus G_{d-1,1}$.
Note that the Newton polygon for $G_{1, d-1} \oplus G_{d-1,1}$ is the lowest Newton polygon in dimension $d$ with $p$-rank $0$.
Thus there is at least one component of the $p$-rank $0$ stratum of $\CM_d$ 
such that the generic geometric point of this component represents a curve whose 
Jacobian has $p$-divisible group isogenous to $G_{1, d-1} \oplus G_{d-1,1}$.
The result is then immediate from \cite[Corollary~6.4]{priesCurrent}.
\end{proof}

\bibliographystyle{amsplain}
\bibliography{npfirstbib}

\end{document}